\documentclass[a4paper,12pt]{amsart}

\def\g{\gamma}

\def\e{\epsilon}

\def\o{\omega}

\usepackage{enumerate}
\usepackage{amssymb} 
\usepackage{latexsym} 
\usepackage{amsfonts} 
\usepackage{amsmath} 
\usepackage{eucal} 
\usepackage{bm} 
\usepackage{bbm} 
\usepackage{graphicx} 
\usepackage[english]{varioref} 
\usepackage[nice]{nicefrac} 
\usepackage[all]{xy}

\usepackage{amsthm}

\theoremstyle{plain}
\newtheorem{thm}{Theorem}[section] 
\newtheorem*{thm*}{Theorem} 
 \newtheorem{prop}[thm]{Proposition}
 \newtheorem{lem}[thm]{Lemma}
 
 \newtheorem{remark}[thm]{Remark}

\theoremstyle{definition}

\newtheorem{defn}[thm]{Definition}
\newtheorem{example}[thm]{Example}

\theoremstyle{remark}

\newtheorem*{remark*}{Remark}
\newtheorem*{claim*}{Claim}

\begin{document}

\author{Jesper Funch Thomsen}
\address{Institut for matematiske fag\\ Aarhus Universitet\\ 8000 \AA rhus C,
Denmark} \email{funch@imf.au.dk}

\title[]{A proof of Wahl's conjecture in the symplectic case}

\begin{abstract}
Let $X$ denote a flag variety of type $A$ or type $C$. We construct 
a canonical Frobenius splitting of $X \times X$ which vanishes with
maximal multiplicty along the diagonal. This way we verify
 a conjecture by Lakshmibai, Mehta and Parameswaran 
\cite{LakshmibaiMehtaParameswaran1998} in type $C$, and obtain
a new proof in type $A$. In particular, we 
obtain a proof of Wahl's conjecture in type $C$, and a new proof 
in type $A$. We also present certain cohomological consequences. 
\end{abstract}

\maketitle

\section{Introduction}

Let $X= \nicefrac{G}{P}$ denote a generalized  flag variety over an algebraically closed 
field $k$. Let $\mathcal I$ denote the sheaf of ideals defining the diagonal subvariety
$\Delta(X)$ of $X \times X$. The sheaf of differential 1-forms $\Omega_X$ 
on $X$ 
then equals the  quotient  $\nicefrac{\mathcal I}{\mathcal I^2}$ 
and consequently we have a surjective map 
$$ \mathcal I \rightarrow \Delta_*( \Omega_X),$$
of sheaves on $X \times X$. Whenever $\mathcal L_1$ and 
$\mathcal L_2$ denote locally free sheaves of rank $1$ on $X$ 
we this way obtain an induced map
\begin{equation}
\label{wahl}
{\rm H}^0\big(X \times X, \mathcal I \otimes (\mathcal L_1
\boxtimes \mathcal L_2) \big) \rightarrow 
{\rm H}^0\big(X , \Omega_X \otimes \mathcal L_1
\otimes \mathcal L_2 \big).
\end{equation}
In \cite{Wahl1991} J. Wahl conjectured that the map 
(\ref{wahl}) is surjective whenever $\mathcal L_1$ and
$\mathcal L_2$ are ample. In characteristic zero (which 
is actually the setup in \cite{Wahl1991}) this conjecture
has been proved by S. Kumar in \cite{Kumar1992} 
using a case by case approach. At the same time Kumar
remarked that a unified approach would be more satisfactory 
and suggestion the notion of Frobenius splitting as a feasible 
tool. This suggestion was later studied in \cite{LakshmibaiMehtaParameswaran1998}.

In \cite{LakshmibaiMehtaParameswaran1998} the characteristic
$p>0$ version of Wahl's conjecture was studied. Here it was 
observed that Wahl's conjecture would follow if $X \times X$ admits
a Frobenius splitting vanishing with maximal multiplicity along the
diagonal. The latter statement was proved to be equivalent to the
existence of a Frobenius splitting of the blow-up of $X \times X$ 
at $\Delta(X)$ which is compatible with the exceptional locus;
i.e. with the projectivized  tangent bundle to $X$. This statement was then 
conjectured (from now on called the LMP-conjecture) to be 
satisfied for generalized flag varieties.

As a first approximation to the LMP-conjecture the paper 
\cite{LakshmibaiMehtaParameswaran1998} gave a construction
of a Frobenius splitting of $X$ vanishing with maximal multiplicity 
along a point. The LMP-conjecture itself was only handled for 
$\nicefrac{{\rm SL}_n(k)}{P}$ when $n \leq 6$. Subsequently the 
LMP-conjecture was verified for  Grassmannians by
Mehta and Parameswaran \cite{MehtaParameswaran1997}, for symplectic
and orthogonal Grassmannians by Lakshmibai, Raghavan and Sankaran
\cite{lakshmibaietal2009} and by Brown and Lakshmibai for minuscule
$\nicefrac{G}{P}$ \cite{BrownLakshmibai2008}. In all these cases the 
proof relied on the fact that a certain "canonical" and well studied 
Frobenius splitting of $X \times X$ vanished with maximal multiplicty 
along the diagonal. At the same time it was observed that this "canonical"
Frobenius splitting would not work in general. Recently Lauritzen and 
Thomsen \cite{LT} introduced a new Frobenius splitting of $X \times X$,
for $X= \nicefrac{{\rm SL}_n(k)}{B}$, with the desired vanishing property along 
the diagonal. As a consequence the LMP-conjecture is now verified for any
generalized flag variety of the form $\nicefrac{{\rm SL}_n(k)}{P}$. The 
paper [loc.cit] also verified the natural generalization of the LMP-conjecture 
for arbitrary  Kempf varieties in  $\nicefrac{{\rm SL}_n(k)}{B}$. 

In the present paper we obtain the Frobenius splitting from \cite{LT}
in a new way. This reveals that the Frobenius splitting of [loc.cit] is 
actually a canonical Frobenius splitting. Moreover, it also reveals its 
relation to the above mentioned Frobenius splitting of $\nicefrac{{\rm SL}_n(k)}{B}$ 
vanishing with maximal multiplicty along the point $eB$
 \cite{LakshmibaiMehtaParameswaran1998}. It turns out 
that this approach generalizes naturally to the case 
$X=\nicefrac{{\rm Sp}_{2m}(k)}{P}$ and thus we also obtain
a proof of the LMP-conjecture in the symplectic case . The
latter should be considered as the main result of this paper. 
Actually, to some extend our approach also generalizes to
the orthogonal case, but for some mysteries reasons  
certain root combinatorics prevents us from concluding 
the LMP-conjecture in this case.

\section{notation}

\subsection{Varieties}

We work over a fixed algebraically closed field $k$ of positive 
characteristic $p>0$. By a variety we mean a reduced 
scheme of finite type over $k$; in particular, a variety need
not be irreducible.  The ring of global regular functions on $X$ 
is denoted by $k[X]$.

\subsection{Group setup}

We fix an integer $n$ and let $G={\rm SL}_n(k)$ denote the group of 
$n \times n$-matrices of determinant $1$. By $B$ we  
denote the Borel subgroup of $G$ consisting of upper triangular 
matrices and let $T \subset B$ denote the maximal torus of 
diagonal matrices.
The set of upper (resp. lower) triangular unipotent matrices in 
$G$ is denoted by $U$ (resp. $U^-$). By $I_n$ we denote 
the identity matrix in $G$.

\subsection{Characters and roots}

The set $X^*(T)$ of $T$-characters is a free abelian 
group of rank $n-1$. If we, for $i=1, \dots, n$, let 
$\epsilon_i$ denote the $T$-character which picks 
out the $i$-th diagonal element, then the set of 
elements $\e_1,\dots, \e_{n-1}$ form a basis for the 
group $X^*(T)$. The set of roots is the 
set of $T$-characters $\epsilon_i-\e_j$, for $i \neq j$. 
A root $\epsilon_i-\e_j$ is said to positive if $i <j$, 
while a positive root is said to be simple if it equals 
$\e_i -\e_{i+1}$ for some $i=1, \dots, n-1$. A 
$T$-character $\lambda$ is said to be dominant if
$$\lambda = \sum_{i=1}^{n-1} \lambda_i \e_i,$$
with coefficients satisfying
$$ \lambda_1 \geq \lambda_2 \geq \cdots 
\geq \lambda_{n-1} \geq 0.$$
The fundamental characters $\omega_1, \dots, 
\omega_{n-1}$ are then the dominant elements
$$ \omega_i = \sum_{s=1}^i \e_s.$$

\subsection{Induced modules}
 
The group of $B$-characters will be denoted by $X^*(B)$ 
and will be identified with the group of $T$-characters 
$X^*(T)$. Any $B$-character $\lambda$ defines an {induced} $G$-representation 
$$ {\rm Ind}_B^G(\lambda) = \{  f :  G \rightarrow k :  f(gb) = 
\lambda(b)^{-1} f(g), ~~ {\text{for} }~~b \in B , ~g \in G \}.$$
The vectorspace structure of ${\rm Ind}_B^G(\lambda)$ is defined in 
the obvious way while the action of $G$ is defined through left-translation.
To simplify notation
we will also use the notation ${\rm H}^0 (\lambda)$ for 
the $G$-module ${\rm Ind}_B^G(-\lambda)$.
Then 
${\rm H}^0 (\lambda)$ is nonzero if and only if $\lambda$ is
dominant. Moreover, in this case ${\rm H}^0 (\lambda)$
contains a unique $U^-$-invariant line which is called the 
lowest weight space of $ {\rm H}^0 (\lambda)$. The 
$T$-character associated to the lowest weight space equals 
$-\lambda$ . The $B$-representation associated to $-\lambda$
is denoted by $k_{-\lambda}$ and is related to ${\rm H}^0 (\lambda)$
under the $B$-equivariant projection map 
$$ {\rm p}_{\lambda} :  {\rm H}^0 (\lambda) \rightarrow k_{-\lambda},$$
$$f \mapsto f({I_n}).$$
When $\lambda$ and $\mu$ are both $B$-characters 
there is a natural multiplication map 
$$ {\rm H}^0 (\lambda) \otimes {\rm H}^0 (\mu)
\rightarrow {\rm H}^0 (\lambda + \mu).$$
It is a central fact (see e.g \cite[Thm.3.1.2]{BrionKumar2005}) that this map is surjective
in case $\lambda$ and $\mu
$ are both dominant.

\subsection{The fundamental representations}
\label{fund}

The modules ${\rm H}^0(\omega_i)$ induced by the 
fundamental characters $\omega_i$, $i=1,\dots,n-1$,
are called fundamental representations of $G$. 
The fundamental representations are simple $G$-modules. 
Moreover, ${\rm H}^0(\omega_{n-i})$ is the module
dual to   ${\rm H}^0(\omega_i)$.
For a 
fixed integer $i$ and a collection ${\bf a} = 
(a_1, a_2, \dots, a_i)$ of increasing integers
$$  1 < a_1 < a_2 < \cdots < a_i \leq n,$$
we let 
$$ v_{ \bf a} : G \rightarrow k,$$
denote the map which takes a matrix $g$ in $G$ to the minor 
determined by the row numbers $a_1,a_2,
\dots, a_i$ and the column numbers $1,2,\dots, i$. Then 
$v_{\bf a}$ is easily seen to be an element of 
${\rm H}^0(\omega_i)$. In fact, the collection 
of elements 
$v_{\bf a}$ form a basis for the fundamental 
representation ${\rm H}^0(\omega_i)$ (cf. 
(\ref{iso1}) below).
In this notation the lowest
weight space of ${\rm H}^0(\omega_i)$ is generated 
by $ v_{(1,2,\dots,i)}$.

\subsection{Exterior products}
\label{ext}
Let $V=k^n$ denote the standard representation 
of $G$ and let $V^*$ denote its dual. We let 
$e_1, \dots, e_n$ denote the standard basis of 
the vector space $V=k^n$, and let $e_1^*,
\dots, e_n^*$ denote the dual basis of 
 $V^*$. As a basis for the $i$-th exterior products 
$\wedge^i V$ we  choose the set of elements 
$$ e_{\bf a} = e_{a_1}
\wedge \cdots \wedge e_{a_i}
\in \wedge^i V,$$
where ${\bf a}=(a_1, \dots,a_i)$ denotes 
 an increasing sequence 
$$1 \leq a_1 < a_2 < \cdots < a_i \leq n,$$
of integers. Similarly we consider 
$$ e_{\bf a}^* =  e_{a_1}^*
\wedge \cdots \wedge e_{a_i}^*
\in \wedge^i V^*,$$
as a basis of $\wedge^i V^*$. We consider 
$\wedge^i V^* $ and  $\wedge^i V $ 
as dual vectorspaces 
by applying the map 
\begin{equation}
\label{dual0}
\wedge^i V^* \otimes \wedge^i V 
\rightarrow k,
\end{equation}
$$ (f_1 \wedge \cdots \wedge f_i) \otimes 
(v_1 \wedge \cdots \wedge v_i) \rightarrow 
{\rm det}(f_s(v_t))_{1 \leq s,t \leq i},$$
In particular,  the basis  $ \{ e_{\bf a}^* \}$ 
of $ \wedge^i V^*$ is dual to the 
basis $ \{ e_{\bf a} \}$ of  $ \wedge^i V$.
We then identify ${\rm H}^0(\omega_i)$ with 
$ \wedge^i V$ using the $G$-equivariant 
isomorphism
\begin{equation} 
\label{iso1}
\wedge^i V^* \xrightarrow{\simeq} {\rm H}^0(\omega_i),
\end{equation}
defined by 
$$ \mathfrak{v} \mapsto \big( g \mapsto {\mathfrak v} (g 
 e_1 \wedge \cdots \wedge e_i) \big).$$
In particular, we identify $e_{\bf a}^*$ with $v_{\bf a}$.

\subsection{The symplectic group}

When $n=2m$ is even we
let $\overline{G}={\rm Sp}_{2m}(k)$ denote the symplectic 
group of rank $m$. We consider ${\overline G}$ as the 
subgroup of elements $g$ in $G={\rm SL}_{n}(k)$ which 
satisfies the relation 
\begin{equation}
\label{rel1} 
g^T A g=A,
\end{equation} 
where $A$ denotes the matrix
$$A= \begin{pmatrix}
0 & {I_m} \\
-{I_m} & 0 \\
\end{pmatrix},
$$
and ${I_m}$ denotes the identity matrix of size $m$.
The Borel subgroup of ${\overline{G}}$ consisting of 
the upper triangular matrices contained in $\overline G$
 will be denoted by $\overline{B}$ while the
unipotent part of $\overline{B}$ will be denoted by
$\overline{U}$. Moreover, we let $\overline{U}^-$ 
denote the set of lower triangular unipotent 
elements in $\overline{G}$. Finally we let $\overline{T}$ denote
the maximal torus in ${\overline{G}}$ consisting of the set of diagonal 
matrices within ${\overline{G}}$. 

\subsection{Symplectic weights, roots and fundamental characters}
\label{restrict2} 

When $\lambda$ is a $T$-character we let $\overline{\lambda}$
denote its restriction to $\overline{T}$.   The group ${\overline{T}}$ 
consists of the set of diagonal matrices in $G$ satisfying that the 
$i$-th diagonal entry, for   $i=1,\dots, n$,  is the inverse of the 
$(n+1-i)$-th diagonal entry.  In particular, $\overline \e_1,
\overline \e_2 ,\dots, \overline \e_m$  form a basis for the group
of characters $X^*(\overline T)$ of $\overline T$. Moreover, any 
root $\epsilon_i - \epsilon_j$ of $G$ restricts to a root
$\overline \e_i - \overline \e_j$ of $\overline G$; in particular 
this restriction is nonzero. The fundamental characters of $\overline G$  are the restrictions
 $\overline{\omega}_i$, $i=1, \dots,m$, of the first $m$ fundamental 
characters for ${G}$.  

We identify the characters of $\overline B$ with the characters of 
$\overline T$. Thus any $\overline T$-character $\overline \lambda$ 
defines an induced $\overline G$-module  :
\begin{equation}
\label{induced}
{\rm Ind}_{\overline{B}}^{\overline{G}}(\overline{\lambda})
= \{  f :  \overline{G} \rightarrow k :  f(gb) = 
\overline \lambda(b)^{-1} f(g), ~~ {\text{for} }~~b \in \overline{B} , ~g \in 
\overline{G} \},
\end{equation}
as in the case of $G$. We also use the short notation 
${\rm H}^0 (-\overline{\lambda})$
for the $\overline{G}$-module in  (\ref{induced}). 
Then any $T$-character $\lambda$ defines  
a $\overline{G}$-equivariant restriction map
\begin{equation} 
\label{RES}
{\rm H}^0(\lambda) \rightarrow
{\rm H}^0 (\overline{\lambda}).
\end{equation}
As (positive) roots of $G$ restricts to (positive) roots
of $\overline G$ the morphism (\ref{RES})  will map 
the lowest weight space (as a $G$-module)
of ${\rm H}^0(\lambda)$ isomorphically to the 
lowest weight space of ${\rm H}^0 (\overline{\lambda})$.
In fact,  any $\overline T$-weight $\overline{\mu}$ of 
$ {\rm H}^0(\lambda)$ distinct from $-\overline{\lambda}$
must satisfy 
that $\overline{\mu}+ \overline{\lambda}$ is a nonzero 
sum of positive roots of $\overline G$. 

\subsection{The Steinberg module}

Let $\rho = \omega_1 + \cdots + \omega_{n-1}$ denote the sum of
the fundamental characters of $G$. The  Steinberg module 
${\rm St}$ is by definition the $G$-module ${\rm H}^0 \big((p-1) \rho \big)$. The 
Steinberg module is an irreducible and self-dual $G$-module. Similarly
we may  define the Steinberg module in the symplectic case to 
be the $\overline{G}$-module $\overline{\rm St}$ induced from the
$(p-1)$-th multiple of the sum 
$ \overline \omega_1 + \cdots +\overline \omega_m,$
of the fundamental characters of $\overline{G}$. Also
in this case ${\overline {\rm St}}$ will be an irreducible and selfdual $\overline G$-module.

\subsection{The dualizing sheaf}
\label{dualizing}

The dualizing sheaf $\omega_{\nicefrac{G}{B}}$ on 
$\nicefrac{G}{B}$ is a $G$-linearized 
line bundle on $\nicefrac{G}{B}$ associated to the 
$B$-character $2 \rho$.  As a consequence the set 
of global sections of a power $\omega_{\nicefrac{G}{B}}^N$ 
of $\omega_{\nicefrac{G}{B}}$ may be $G$-equivariantly 
identified with 
the induced module ${\rm H}^0\big(-2 N \rho \big)$.
A similar result is true in the symplectic case once 
$\rho$ is substituted by the sum $\overline \omega_1+
\cdots + \overline \omega_m$ of the fundamental characters.
As we will see in the next section the case $N=1-p$ is related
to the notion of Frobenius splittings.

\section{Frobenius splitting}

Let $X$ denote a variety. The {\it absolute Frobenius morphism} on $X$ is the map of
schemes $F_X : X \rightarrow X$ which on the level of points is the identity
map and as a map of structure sheaves 
$$ F^{\sharp}_X : \mathcal O_X \rightarrow (F_X)_* \mathcal O_X ,$$
is the $p$-th power map. 

\begin{defn}(\cite{MehtaRamanathan1985})
The variety $X$ is said to admit a {\it Frobenius splitting} (or 
an F-splitting in short) if the map
of sheaves 
\begin{equation}
 F^{\sharp}_{X} : \mathcal O_{X} \rightarrow (F_{X})_* \mathcal O_X ,
\end{equation}
admits an $\mathcal O_{X}$-linear splitting, i.e. an element 
$$ s  \in {\rm End}_F(X)  := {\rm Hom}_{\mathcal O_{X}} \big(  (F_{X})_* \mathcal O_X ,
\mathcal O_{X} \big)$$
such that the composition $s \circ F^{\sharp}_{X}$ is the identity 
map on $\mathcal O_{X}$.

\end{defn}

\subsection{A local description}

When $X$ is a smooth variety we let $\omega_X$ denote the dualizing 
sheaf of $X$. By duality, for the finite morphism $F_{X}$, there is
a canonical $k$-linear identification
\begin{equation}
\label{cartier}
{\rm End}_F(X)^{(1)} \simeq {\rm H}^0 ( X , \omega_X^{1-p} ),
\end{equation}
where we use the notation ${\rm End}_F(X)^{(1)}$ for the set
$ {\rm End}_F(X)$ with $k$-structure twisted by the $p$-th root
map on $k$.
In particular, we may think of elements in
$ {\rm End}_F(X)$  as global sections of a linebundle. 
For a point $P \in X$ we may describe the identification 
(\ref{cartier}) locally as follows (see \cite[Prop.5]{MehtaRamanathan1985}): 
let $x_1,x_2, \dots, x_N$ 
denote a system of local coordinates at a point  $P$ of $X$. Then the section 
\begin{equation}
\label{dual}
{x_1^{a_1} x_2^{a_2} \cdots x_N^{a_N}}
( dx_1 \wedge dx_2 \wedge \cdots \wedge dx_N)^{1-p},
\end{equation}
of $\omega_X^{1-p}$ 
corresponds to the element in ${\rm End}_F(X)$ satisfying
$$ x_1^{b_1} x_2^{b_2} \cdots x_N^{b_N}
\mapsto  x_1^{\frac{a_1+b_1+1}{p}-1 }
 x_2^{\frac{a_2+b_2+1}{p}-1 } \cdots 
 x_n^{\frac{a_N+b_N+1}{p}-1 },$$
where we consider the right hand side to be zero if 
one of the rational exponents $\frac{a_i+b_i+1}{p}-1$
is not an integer. 

\begin{example}
\label{ex1}
Consider the case $X = \mathbb A^N$ in which 
$\omega_X$ is isomorphic to $\mathcal O_X$. 
Thus elements in ${\rm End}_F(X)$ correspond
to polynomials in 
$$ k[X] =  k[X_1, \dots, X_N].$$
  This correspondence is uniquely determined up 
to nonzero constants. 
Hence we may speak of an element $g$ in 
$k[X_1, \dots, X_N]$ as defining a Frobenius 
splitting of $\mathbb A^N$ (at least up to 
a nonzero constants). A necessary, but not
sufficient, condition for this is that the 
monomial
$$( X_1 X_2 \cdots X_N)^{p-1},$$
appears in $g$ with nonzero coefficient. To 
obtain a sufficient condition we need to add 
the conditions that all monomials of the form 
$$( X_1 X_2 \cdots X_N)^{p-1} M^{p},$$
for $M$ a nonzero monomial in $k[X_1,\dots,
X_N]$, should not appear in $g$.
\end{example}

\subsection{Compatible splitting}

Let $Y$ denote a closed subvariety of $X$ and let $\mathcal I_Y$
denote the associated sheaf of ideals in $\mathcal O_X$. 

\begin{defn}
An element  $s  \in {\rm End}_F(X)$ is said to 
be compatible with $Y$ if 
$$ s( (F_{X})_* \mathcal {I}_{Y} ) \subset \mathcal {I}_{Y}.$$
If, moreover, $s$ defines a Frobenius splitting of $X$ then we
say that $Y$ is compatibly Frobenius split in $X$.
The set of elements in ${\rm End}_F(X)$ which are 
compatible with $Y$ will be denoted by ${\rm End}_F(X,Y)$.
\end{defn}

\subsection{Multiplicities}
When $X$ and $Y$ are both smooth varieties we may use the 
local description in (\ref{cartier}) to obtain a criterion for the
compatibility of $Y$ : let $P \in Y$ and choose a 
system of local coordinates at $P$  in such a way that the ideal of 
$Y$ is generated by $x_1, \dots, x_d$ (with $d$ denoting the 
codimension of $Y$ in $X$). Then the element (\ref{dual}) 
will map the ideal $(x_1, \dots, x_d)$ to itself if one of the following
conditions are satisfied
\begin{itemize}
\item $a_i \geq p$ for some $1 \leq i \leq d$. 
\item $a_1=a_2 = \cdots = a_d = p-1$. 
\end{itemize}
In particular, this is the case if (\ref{dual}) vanishes
with multiplicity at least $(p-1) d$ along $Y$. This
leads to the following definition

\begin{defn}
Let $X$ be a smooth variety and let $Y$ denote a closed 
smooth subvariety. We say that an element $s$ of ${\rm End}_F(X)$ 
vanishes with multiplicity $m$ along $Y$, if the same is true  
for the associated global section of the line bundle $\omega_X^{1-p}$. 
\end{defn}

It then follows

\begin{lem}
\label{com}
Assume that $X$ and  $Y$ are smooth varieties and that 
$Y$ is of codimension $d$ in $X$. Let $s$ be an element  
of ${\rm End}_F(X)$ vanishing with multiplicity at least
$d (p-1)$ along $Y$. Then $s$ is compatible with $Y$. 
\end{lem}

When $s$ is assumed to be a Frobenius splitting of $X$ 
a vanishing multiplicity along $Y$ of size $d (p-1)$ is said 
to be maximal. This notion is explained  by the 
following result

\begin{prop}(\cite[Sect.2]{LakshmibaiMehtaParameswaran1998})
\label{maxmult}
Let $X$ be a smooth variety and 
let $Y$ denote a smooth closed subvariety of codimension $d$. 
Let $s$ denote a Frobenius splitting of $X$. Then $s$ vanishes 
with multiplicity at most $(p-1) d$ along $Y$.  Moreover, if $s$
vanishes with {\it maximal multiplicity} 
then  the 
blow-up of $X$ at $Y$ admits a Frobenius splitting which is
compatible with the exceptional divisor. 
\end{prop}

The maximal vanishing multiplicty is particular important 
in the case of the diagonal subvariety ${\rm diag}(X)$ in the
product $X \times X$. With $\boxtimes$ denoting the exterior
tensorproduct on $X \times X$ we have the following result 

\begin{thm}
\label{wahlsurj}
Let ${X}$ denote a smooth projective variety. Let
$\mathcal L$ denote an ample line bundle on $X$,
and  $\mathcal M_1$ and $\mathcal M_2$
denote globally generated line bundles on $X$. 
Define $\mathcal L_i = \mathcal L \otimes \mathcal M_i$,
for $i=1,2$.
Let further $\mathcal I$ denote the sheaf of ideals defining
the diagonal subvariety $ {\rm diag}(X)$ in $X \times X$. If
$X \times X$ admits a Frobenius splitting vanishing with 
maximal multiplicity along the diagonal, then the natural 
restriction map
$${\rm H}^0\big(X \times X, \mathcal I \otimes
(\mathcal L_1 \boxtimes  \mathcal L_2) \big)
\rightarrow
{\rm H}^0\big(X, \Omega_X^1 \otimes 
\mathcal L_1 \otimes \mathcal L_2 \big),$$
is surjective.
\end{thm}

Theorem \ref{wahlsurj} was proved in 
\cite[Sect.3]{LakshmibaiMehtaParameswaran1998}
in the case of a generalized flag variety $X$. 
The general form of Theorem \ref{wahlsurj} was 
obtained in \cite[Thm.6.1]{LT}. Notice that 
when $X$ is a generalized flag variety (in 
fact, any Schubert variety \cite[Sect.6]{LT})) 
then any pair $\mathcal L_1$, $\mathcal L_2$ of 
ample line bundles on $X$ will be of the form
described in Theorem \ref{wahlsurj}. In 
particular, in this case Wahl's conjecture is a 
consequence of Theorem \ref{wahlsurj}. 

As a consequence of Proposition \ref{maxmult}
the existence of a Frobenius splitting of $X \times X$
vanishing with maximal multiplicty along the
diagonal will imply that the projectivized tangent 
bundle of $X$ admits a Frobenius splitting. 
We therefore recall the following vanishing 
result

\begin{thm}(\cite[Prop.6.8]{LT})
Let ${X}$ denote a smooth projective variety
and assume that the projectivized tangent bundle 
on $X$ admits a Frobenius splitting.
Let $\mathcal L$ (resp. $\mathcal M$) 
denote a very ample (resp. globally generated) 
line bundle on $X$ and let $j> 0$ denote an 
integer. Then 
\begin{equation}
\notag
{\rm H}^i\big( Z , S^j \Omega_{Z} \otimes {\mathcal L}^{2j}
\otimes \mathcal M \big) = 0 ~, \text { for } i>0.
\end{equation}
\end{thm} 

\

We refer to \cite{LT} for further related vanishing 
results.

\subsection{The evaluation map}
\label{evx}

Any element  in  ${\rm End}_F(X)$ may be evaluated 
at the constant function $1 \in k[X]$. This defines a 
map
$$ {\rm ev_X} : {\rm End}_F(X) \rightarrow k[X],$$
$$ s \mapsto s(1).$$
Thus $s$ in  ${\rm End}_F(X)$
defines a Frobenius splitting of $X$ if and only if 
${\rm ev}_X(s) = 1$. In particular, the condition 
that $s$ is a Frobenius splitting may be checked 
on an open dense subset of $X$; i.e. if $V$ is
an open dense subset of $X$ and if the restriction 
of $s$ to $V$ defines a Frobenius splitting of $V$,
then $s$ defines a Frobenius splitting of $X$. 
In case $k[X]=k$, e.g. if  $X$ is complete and irreducible,
any element $s\in {\rm End}_F(X)$  with nonzero evaluation 
${\rm ev}_X(s)$ 
will thus define a Frobenius splitting of $X$ up to a nonzero constant.

When $X$ is smooth we may use the evaluation 
map ${\rm ev}_X$  and the identification (\ref{cartier}) 
to define a $k$-linear map ($k$ is perfect)
$$ {\rm H}^0\big(X, \omega_X^{1-p} \big)
\rightarrow k.$$
This map will, by abuse of notation, also be denoted
by ${\rm ev}_X$ and will  also be called the evaluation
map. By the local description (\ref{dual}) we may 
(locally) consider ${\rm ev}_X$ as the map which picks 
out the coefficient in front of
$$(x_1 x_2 \cdots x_N)^{p-1} (dx_1 \wedge \cdots 
\wedge dx_N)^{1-p},$$ 
 in the local monomial expansions of an element 
$s \in {\rm H}^0\big(X, \omega_X^{1-p} \big)$
relative to a system of local coordinates $x_1, \dots, 
x_N$ at a point $P$ of $X$. In particular, ${\rm ev}_X(s)$
is nonzero if this coefficient is nonzero for some (and
hence any) point $P$.
 
\subsection{Canonical Frobenius splittings}

Let now $X$ denote a $G$-variety, and consider 
the vector space ${\rm End}_F(X)$ as a $G$-module 
in the natural way. A $G$-equivariant morphism 
$$ \phi : {\rm St} \otimes {\rm St} 
\rightarrow
{\rm End}_F(X)^{(1)},$$
is said to be a canonical Frobenius splitting of $X$
if the image of $\phi$ contains a Frobenius splitting
of $X$. The notion of a canonical Frobenius 
splitting of a ${\overline G}$-variety is defined 
similarly.

\section{Residual normal crossing}

In the following we will consider a polynomial ring
$k[X_1, \dots, X_N]$ and a totally ordered subset 
$J_1$ of the set of variables $X_i$.  The order on $J_1$
will be denoted by $ \leq_1$. 

\begin{defn}
An element $g$ in $k[X_1, \dots, X_N]$ is  
said to be a {\it residual normal crossing relative to 
the ordered subset $J_1$}, if
the following recursive conditions are 
satisfied
\begin{enumerate}
\item If $J_1$ is empty then $g$ is a nonzero constant.
\item If $J_1$  is nonempty and $X_j$ is the minimal element in 
$J_1$ then $X_j$ divides $g$.  Moreover, the 
evaluation 
$(\frac{g }{ X_{j}})_{|X_j=0}$
of the quotient $\frac{g }{ X_{j}}$  at $X_j=0$,
is then a residual normal crossing relative to the subset 
$ J_1 \setminus \{X_j \}$, with the induced 
total order. 
\end{enumerate}
\end{defn}

Notice that a necessary (but not sufficient) condition for 
$g$ to be a residual normal crossing relative to 
$J_1$,  is that the monomial 
$$  \prod_{j \in J_1} X_j,$$
appears in $g$ with nonzero coefficient. In fact, even 
more general results like this are true (see Proposition
\ref{fres} below).

When $(J_1, \leq_1)$ and $(J_2,\leq_2)$ denote disjoint 
totally ordered subsets  of the set of variables, then we 
let $\leq_{1,2}$ denote the total order on the union 
$J_1 \cup J_2$ defined by : 
$ X_a \leq_{1,2} X_b $ if one of the following conditions are
satisfied
\begin{itemize}
\item $X_a,X_b \in J_1$ and $X_a \leq_1 X_b.$ 
\item  $X_a,X_b \in J_2$ and $X_a \leq_2 X_b.$ 
\item $X_a \in J_1$ and $X_b \in J_2$.
\end{itemize}
This explains the notation in the following statement

\begin{lem}
\label{product}
Let $f$ and $g$ denote elements in $k[X_1,\dots,X_N]$. 
Assume that $g$ is a residual normal crossing relative
to a subset $J_1$, and that the evaluation 
$$ \overline{f} = f_{ | \{ X_j= 0 : X_j \in J_1\} }$$
is a residual normal crossing relative to an ordered 
subset $J_2$. Then $J_1$ and $J_2$ are
disjoint and the product $fg$ is a residual normal crossing
relative to the union $J_1 \cup J_2$ ordered by $\leq_{1,2}$.
\end{lem}
\begin{proof}
That $J_1$ and $J_2$ are disjoint is clear as 
$\overline{f}$ does not involve any of the 
variables in $J_1$. In particular, 
the monomial
$$\prod_{X_j \in J_2} X_j,$$
could not appear in $\overline{f}$ unless 
$J_1$ and $J_2$ were disjoint. That $fg$ 
is a residual normal crossing relative to 
the ordered subset $J_1 \cup J_2$ now 
follows directly from the recursive definition. 
\end{proof}

The definition of being a residual normal crossing 
does not involve any of the variables $X_j$ 
outside  $J_1$. In particular, 
we have the following easy, but still useful, 
observation

\begin{lem}
\label{subst}
Let $g \in k[X_1,\dots,X_N]$ be a residual normal 
crossing relative to an ordered subset $J_1$. Assume
that $X_j \notin J_1$ and let 
$$ h= g(X_1, \cdots, X_{j-1}, f, X_{j+1}, \cdots,
X_N),$$
denote effect of substituting the $j$-th variable in $g$
by a polynomial $f \in k[X_1,\dots,X_N]$. Then $h$
is also a residual normal crossing relative to the 
subset $J_1$. 
\end{lem}

\subsection{Frobenius splitting and residual normal 
crossing}
\label{connection}

Consider a complete variety $X$ containing an open 
dense subset $V$ isomorphic to $\mathbb A^N$. 
As explained in Section \ref{evx} an element $s$
of ${\rm End}_F(X)$ defines a Frobenius splitting 
of $X$, up to a nonzero constant, if and only if
the evaluation ${\rm ev}_X(s)$ is nonzero. Moreover,
this condition may be checked locally and means that 
$s$, when expressed in a system of local coordinates $x_1,
\dots, x_N$ at a point $P$, contains the monomial
$$ (x_1 x_2 \cdots x_N)^{p-1} (dx_1 \wedge \cdots
\wedge dx_N)^{1-p},$$
with nonzero coefficient. E.g. we could take $P$ to 
denote the origin (or any other point) in $V = 
\mathbb A^N$. In this case we may also formulate 
the nonzeroness condition of ${\rm ev}_X(s)$ as follows: 
consider, as explain in Example 
\ref{ex1},  the restriction $s_{|V}$, of $s$ to $V$, as 
an element in the coordinate ring
$$ k[V] = k[X_1,X_2, \dots, X_N],$$
of $V$. Then the monomial 
$$(X_1 X_2 \cdots X_N)^{p-1},$$
should appear in $s_{|V}$ 
with nonzero coefficient. A condition like the
latter may be checked using the notion of 
residual normal crossings. This is explained 
by the following result which follows easily 
from the definition.

\begin{prop}
\label{fres}
Let $g \in k[X_1,\dots, X_N]$ denote a residual normal
crossing relative to an ordered subset $J_1$ of the 
variables. Let $d$ denote any positive integer. 
Then the monomial
$$ \prod_{X_j \in J_1} X_j^d,$$
appears with nonzero coefficient in $g^d$.
\end{prop}

\section{Concrete residual normal crossings}
\label{residual}

In this section we will study certain residual normal 
crossings associated to the group  $U^-$. We start by observing 
that $U^-$   is isomorphic to affine 
${n \choose 2}$-space, and choose
coordinate 
functions   $Z_{i,j}$ , $1 \leq j <  i \leq n$, 
in the natural way.

The functions $v_{\bf a}$ on $G$ defined in Section 
\ref{fund} are semi-invariant under translation by $B$
from the right and hence they are uniquely determined 
by their restrictions to $U^-$. It thus makes sense  
to consider $v_{\bf a}$ as functions on $U^-$; i.e. 
as elements in the polynomial ring  
 $k[Z_{i,j}] _{1 \leq j < i \leq n}$.
In the following we will consider the  functions ($1 \leq r \leq n$) 
$$ \mathfrak f_{r} = \prod_{s=1}^{r} 
v_{(\lfloor \frac{s}{2}
\rfloor+1 , \lfloor \frac{s}{2} \rfloor+2,
\dots, s)} \in k[U^-],$$
where $\lfloor \frac{s}{2} \rfloor$ denotes the biggest
integer smaller than $\frac{s}{2}$.
We claim that $\mathfrak f_r$ is a residual normal crossing
relative to an ordered subset of the  variables $Z_{i,j}$. To
make this precise we introduce the following ordering of 
the variables : $ Z_{i,j} \leq Z_{i',j'}$ if either $i+j <
i'+j'$ or if $i+j = i'+j'$ and $i \leq i'$. In the following
any subset of the variables will, unless otherwise stated,
be assumed to be ordered using this order. 

\begin{lem}
\label{res1}
The polynomial $\mathfrak f_r$ is a residual normal crossing 
relative to the subset 
$$\{ Z_{i,j} :  i+j \leq r+1 \},$$ 
of the variables.
\end{lem}
\begin{proof}
We proceed by induction in $r$. As 
$$\mathfrak f_1 = 1,$$
the case $r=1$ is clear. Now assume that  $1 < r \leq n$ and that 
the statement is correct for $r-1$. Notice then
$$ \mathfrak f_r = \mathfrak f_{r-1} \cdot  v_{(\lfloor \frac{r}{2}
\rfloor+1, 
\dots, r)},$$
and that $ v_{(\lfloor \frac{r}{2}
\rfloor+1, 
\dots, r)}$, when 
evaluated at $Z_{i,j}=0$ for $i+j \leq r$, equals 
\begin{equation}
\label{factor2}
 \prod_{\stackrel{1 \leq j < i \leq n, }{ i+j = r+1}} Z_{i,j},
\end{equation}
up to a sign. We may thus apply Lemma \ref{product} to conclude that $\mathfrak f_r$
is a residual normal crossing relative to $Z_{i,j}$, for $i+j \leq r+1$.
This ends the proof.
\end{proof}

Let  $\mathfrak w_0$ in ${\rm GL}_n(k)$ denote the anti-diagonal 
matrix
whose $(i,j)$-th entry equals $1$ if $i+j = n+1$ and equals
$0$ otherwise. Then 
\begin{equation} 
\label{tau}
\tau : {\rm SL}_n(k) \rightarrow {\rm SL}_n(k),
\end{equation}
$$ \tau(g) = \mathfrak w_0 g^T \mathfrak w_0,$$
defines an automorphism of the variety $G$ of order 
two which leaves $U^-$ invariant. The map $\tau$ 
is not a group automorphism but composing $\tau$ 
with the inverse map 
$$ \iota : G \rightarrow G, ~ ~ \iota(g) = g^{-1},$$
we obtain an involution $\sigma = \tau \circ \iota$
of $G$ leaving $B$ invariant. In the following we let
$\mathfrak g_r \in k[U^-]$ denote the composition
$\mathfrak f_r \circ \tau$.

\begin{lem}
\label{res3}
The polynomial $\mathfrak g_r$ is a residual normal 
crossing relative to the reverse order on the subset 
$$\{ Z_{i,j} :  i+j \geq 2n+1-r \},$$ 
of the variables.
\end{lem}
\begin{proof}
This follows directly from Lemma \ref{res1} once it is observed
that the effect of $\tau$ on $k[U^-]$ is given by the order reversing 
map $Z_{i,j}  \mapsto Z_{n+1-j, n+1-i}$.
\end{proof}

\begin{prop}
\label{res4}
There exists a total order on the full set of variables
$Z_{i,j}, 1 \leq j < i \leq n$, such that 
$$ \mathfrak f_n \mathfrak g_{n-1} $$
is a residual normal crossing relative to this ordered
set.
\end{prop}
\begin{proof}
This follows by applying Lemma \ref{res1}, Lemma
\ref{res3} and Lemma \ref{product}.
\end{proof}

\begin{remark}
In \cite{LakshmibaiMehtaParameswaran1998} a Frobenius 
splitting of $\nicefrac{G}{B}$  vanishing with maximal
multiplicity along the point $eB$ was constructed. 
The restriction of this Frobenius splitting to $U^-$ 
coincides with the $(p-1)$-th power of the composition 
$(\mathfrak f_n \mathfrak g_{n-1}) \circ \iota$
(see \cite[Remark $A_n$ (i)]{LakshmibaiMehtaParameswaran1998} 
and Lemma \ref{twist} below).  The idea 
of the present paper is to induce this Frobenius splitting into a
 Frobenius 
splitting of $\nicefrac{G}{B} \times \nicefrac{G}{B}$
vanishing with maximal multiplicity along the diagonal.
A similar remark applies in the symplectic case.  
\end{remark}

\section{Tensorproducts of fundamental representations}

In this section we will consider tensorproducts of the 
fundamental representations of $G$. To simplify notation 
we let $\omega_n$ and $\omega_0$ denote the trivial
$T$-characters. 
Fix integers $0 \leq i,j \leq n$, and let $[i+j]_n$ 
denote either $i+j$ if $i+j \leq n$, and 
$i+j-n$ otherwise. We want to construct a
$G$-equivariant nonzero morphism 
\begin{equation}
\label{embedding}
\eta_{i,j} : {\rm H}^0(\omega_{[i+j]_n}) \rightarrow 
{\rm H}^0(\omega_{i})  \otimes 
{\rm H}^0(\omega_{j}). 
\end{equation} 
We divide this construction into to cases 

\subsection{The case $i+j \leq n$}
By (\ref{iso1}) it suffices to construct a nonzero
$G$-equivariant morphism
\begin{equation}
\label{dualmult}
 \wedge_{i,j} : \wedge^{i+j} V^*
\rightarrow \wedge^{i} V^* \otimes
\wedge^{j} V^*.
\end{equation}
The latter is defined to be the dual of the product
map
\begin{equation} 
\wedge^i V \otimes \wedge^j V \rightarrow
\wedge^{i+j} V.
\end{equation}
With notation as in Section \ref{fund} and Section
\ref{ext} it is then an easy exercise to check that 
\begin{equation}
\label{shuffle}
\eta_{i,j}(v_{\bf a}) = 
\sum_{\g \in {\rm Sh}_{i,j}} {\rm sgn}(\g) ~~
v_{(a_{\g(1)}, \cdots, a_{\g(i)})} \otimes
v_{(a_{\g(i+1)}, \cdots, a_{\g(i+j)} )},
\end{equation}
where we use the notation ${\rm Sh}_{i,j}$ to
denote the set of $(i,j)$-shuffles,
i.e. the permutations $\g$ of $1,2\dots,i+j$ satisfying 
$\g(1) < \g(2) < \cdots < \g(i)$ and 
$\g(i+1) < \g(i+2) < \cdots < \g(i+j)$.
In particular, we note

\begin{lem}
\label{eval}
The composition of  $\eta_{i,j}$, for $i+j \leq n$, with 
the $B$-equivariant map 
\begin{equation}
\label{ev}
{\rm p}_{\omega_i} : {\rm H}^0\big( \omega_i \big) \rightarrow k_{-\omega_i},
\end{equation}
defines a $B$-equivariant morphism
$$  {\rm H}^0(\omega_{i+j}) \rightarrow k_{-\omega_i}
 \otimes  {\rm H}^0(\omega_{j}) ,
$$
given by 
$$v_{\bf a} \mapsto
\begin{cases}
1 \otimes v_{(a_{i+1}, \dots, a_{i+j}) }  & \text{if  $a_s=s$ for $s=1,\dots,i$}, \\
0 & \text{else},
\end{cases}
$$
when ${\bf a} = (a_1, \dots, a_{i+j})$. 
\end{lem}
\begin{proof}
This follows from (\ref{shuffle}) once it is observed that the natural 
morphism (\ref{ev}) is given by 
$$ v_{\bf b} \mapsto 
\begin{cases}
1 & \text{if {\bf b} = (1,2,\dots, i)}, \\
0 & \text{else},
\end{cases}.
$$
for ${\bf b} = (b_1,\dots, b_i)$. 
\end{proof}

\subsection{The case $i+j >n$}

To define $\eta_{i,j}$ when $i+j >n$ we will use the
involution $\sigma$ of $G$ defined in Section \ref{residual}.  
Any element $f : G \rightarrow k$ in 
${\rm H}^0(\omega_s)$ will by composition
$f \circ \sigma$ with $\sigma$ define an element in 
${\rm H}^0(\omega_{n-s})$. This way we obtain  
a bijection between ${\rm H}^0(\omega_s)$ and
${\rm H}^0(\omega_{n-s})$.  Alternatively this
may be formulate by saying that $\sigma$ defines 
a $G$-equivariant isomorphism
\begin{equation}
\label{sigma2}
{\rm H}^0(\omega_s)^\sigma \simeq
{\rm H}^0(\omega_{n-s}),
\end{equation}
where ${\rm H}^0(\omega_s)^\sigma$ denotes 
the $G$-module 
which as a vector space is 
${\rm H}^0(\omega_s)$ but where the 
$G$-action is twisted by $\sigma$. As  
 $\eta_{n-i,n-j}$ has already been defined
above we may now define
$$  \eta_{i,j} : {\rm H}^0(\omega_{i+j-n}) \rightarrow 
{\rm H}^0(\omega_{i})  \otimes 
{\rm H}^0(\omega_{j}),$$ 
to the $\sigma$-twist of $\eta_{n-i,n-j}$.

In the following we will need more specific information
about the structure of (\ref{sigma2}). Recall that $\iota$ 
denotes the map $\iota(g) = g^{ -1}$ for $g \in 
U^-$. Then

\begin{lem}
\label{twist}
Let $2 \leq i \leq n$ denote an integer. Then 
the functions
\begin{equation}
\label{v3}
v_{(\lfloor \frac{i}{2} \rfloor +1,
 \lfloor \frac{i}{2} \rfloor +2, \dots, 
i)} \circ \iota,
\end{equation}
and
$$  v_{(\lceil \frac{i}{2} \rceil +1,
 \lceil \frac{i}{2} \rceil +2, \dots, 
i)},$$
on $U^-$ coincides up to a nonzero constant.
\end{lem}
\begin{proof}
By concentrating on the first $i$ rows and columns 
we may assume that $i=n$. Now notice that 
\begin{equation}
\label{v1} 
v_{(\lceil \frac{n}{2} \rceil +1,
 \lceil \frac{n}{2} \rceil +2, \dots, n)},
\end{equation}
is a $B \times B$-semiinvariant functions 
in ${\rm H}^0(\omega_{\lfloor \frac{n}{2} 
\rfloor})$; in particular, it must generate the highest 
weight space in  ${\rm H}^0(\omega_{\lfloor 
\frac{n}{2} \rfloor})$. Similarly,
$ v_{(\lfloor \frac{n}{2} \rfloor +1,
 \lfloor \frac{n}{2} \rfloor +2, \dots, n)}$
is a a highest weight vector of 
 ${\rm H}^0(\omega_{\lceil
\frac{n}{2} \rceil})$ and thus by 
(\ref{sigma2}) the composition 
\begin{equation}
\label{v2}
 v_{(\lfloor \frac{n}{2} \rfloor +1,
 \lfloor \frac{n}{2} \rfloor +2, \dots, n)}
\circ \sigma,
\end{equation}
is a highest weight vector in ${\rm H}^0(\omega_{\lfloor 
\frac{n}{2} \rfloor})$. It follows that (\ref{v1})
and (\ref{v2}) coincide up to a nonzero constant 
as functions on $G$. To end the proof  observe
that (\ref{v2}) coincide with (\ref{v3}) (for $i=n$) 
as functions on $G$ and thus also on $U^-$.
\end{proof}

\begin{remark} 
As observed by Shrawan Kumar the above described 
maps $\eta_{i,j}$ are also predicted by the PRVK-conjecture.   
E.g. when $i+j \leq n-1$, the PRVK-conjecture (see e.g.
\cite[Thm4.3.2]{BrionKumar2005})
predicts the existence of a unique map (up to constants)
$$  {\rm H}^0\big( \theta \big) \rightarrow 
{\rm H}^0 \big( \omega_i \big) \otimes
{\rm  H}^0 \big( \omega_j \big) ,$$
where $\theta$ is the dominant $T$-characters in
the Weyl group orbit of the character $\omega_i + w_0 \omega_j$
($w_0$ denotes the element of maximal length in the Weyl group). It 
is an easy exercise to check that $\theta$ coincide with 
$\omega_{ i+j}$.
\end{remark}

\section{Vanishing behaviour of elements in the image of $\eta_{i,j}$}
\label{vanish}

Elements in the image of $\eta_{i,j}$ may be regarded  as 
regular functions on $G \times G$. In this section
we will obtain a lower bound on their order of vanishing 
along the diagonal ${\rm diag}(G)$ of $G \times G$. 
We let $g,h \in G$ denote elements with entries $g_{s,t}$ and $h_{s,t}$, 
$1 \leq s,t \leq n$, respectively, and form the $n \times 2n$ matrix  
\begin{equation}
\label{M}
M(g,h) = 
\begin{pmatrix}
g_{1,1} & \cdots & g_{1, n} & h_{1,1} & \cdots & h_{1,n} \\   
g_{2,1} & \cdots & g_{2, n} & h_{2,1} & \cdots & h_{2,n} \\ 
\vdots  & \cdots & \vdots & \vdots & \cdots & \vdots \\
g_{n,1} & \cdots & g_{n, n} & h_{n,1} & \cdots & h_{n,n} \\ 
 \end{pmatrix}.
\end{equation}

\subsection{The case $i+j \leq n$}

Fix $i$ and $j$ such that $i+j \leq n$ and
an increasing sequence ${\bf a}=(a_1, \dots,a_{i+j})$. 

\begin{lem}
\label{lemma1}
The evaluation of the function $\eta_{i,j}(v_{\bf a})$ at $(g,h)$ 
coincides with the $(i+j)$-minor of $M(g,h)$ determined by the 
row numbers 
$$ a_1,a_2 \dots, a_{i+j},$$
and the column numbers 
$$1, \dots, i ~ ~ ~  \text{ and } ~ ~ ~  n+1, \dots, n+j.$$
In particular,
the element $\eta_{i,j}(v_{\bf a})$ vanishes with multiplicity 
at least ${\rm min}(i,j)$ along the diagonal in $G \times G$.
\end{lem}
\begin{proof}
Recall that $v_{\bf b}$, for ${\bf b}=(b_1,\dots, b_s)$, is defined 
as the map $G \rightarrow k$ which takes a matrix to its 
$s \times s$-minor determined by the row numbers $b_1,\dots, b_s$ 
and the column numbers $1,\dots,s$. Thus, by the Laplace expansion 
theorem, we recognize the right hand side of  (\ref{shuffle})
as the determinant of the square matrix
\begin{equation}
\label{ij-minor}
\begin{pmatrix}
g_{a_1,1} & \cdots & g_{a_1, i} & h_{a_1,1} & \cdots & h_{a_1,j} \\   
g_{a_2,1} & \cdots & g_{a_2, i} & h_{a_2,1} & \cdots & h_{a_2,j} \\   
\vdots  & \cdots & \vdots & \vdots & \cdots & \vdots \\
g_{a_{i+j},1} & \cdots & g_{a_{i+j}, i} & h_{a_{i+j},1} & \cdots & h_{a_{i+j},j} \\   
 \end{pmatrix}.
\end{equation}
This ends the proof of the first part of the statement.

For the vanishing part of the statement we may, by
symmetry, assume that $i \leq j$. On the diagonal 
(i.e. when $g=h$) the matrix $M(g,h)$ contains $i$
pairs of equal columns; column $1$ and $i+1$, 
column $2$ and $i+2$ etc. In particular, if 
we subtract column $s$ from column $n+s$, for
$s=1, \dots,i$, before  calculating the 
$(i+j)$-minor (\ref{ij-minor}), then
the vanishing statement is readily 
apparent.   
\end{proof}

\subsection{The case $i+j > n$}

Now fix $i$ and $j$ such that $i+j >n$ and recall that 
$\eta_{i,j}$ was defined using the $\sigma$-twist of 
$\eta_{n-i,n-j}$. Then Lemma \ref{lemma1} directly implies

\begin{lem}
\label{lemma2-1}
Every element within the image of  $\eta_{i,j}$ will vanish with 
multiplicity at least ${\rm min}(n-i,n-j)$ along the diagonal in 
$G \times G$. 
\end{lem}

\section{Frobenius splitting of $\nicefrac{G}{B} \times \nicefrac{G}{B}$}

In this section we will construct a canonical Frobenius splitting 
$$ \eta :{\rm St} \otimes {\rm St} \rightarrow {\rm End}_F \big 
(\nicefrac{G}{B} \times \nicefrac{G}{B} \big)^{(1)},$$
of $\nicefrac{G}{B} \times \nicefrac{G}{B}$ where any element in 
the image will vanish with multiplicity at least $(p-1) 
{\rm dim}(\nicefrac{G}{B})$ along the diagonal. In particular, 
any element in the image of $\eta$ will be compatible with 
the diagonal ${\rm diag}(\nicefrac{G}{B})$. Actually we 
will construct $\eta$ using the  $G$-equivariant identification  
\begin{equation} 
\label{id1}
{\rm End}_F\big(\nicefrac{G}{B} \times 
\nicefrac{G}{B}\big)^{(1)} \simeq
{\rm H}^0\big(2(p-1) \rho\big)
\otimes {\rm H}^0\big(2(p-1) \rho\big),
\end{equation}
arising from the remarks in Section \ref{dualizing}

\subsection{Vanishing multiplicities}

The map $\eta$ is constructed as the tensorproduct of a collection 
of the maps $\eta_{i,j}$. To describe this construction we 
start by considering the tensorproduct of the maps ($1 \leq s <n$)
 $$ \eta_{ \lfloor \frac{s}{2} \rfloor, \lceil \frac{s}{2} \rceil} :
{\rm H}^0 \big (\omega_s \big) \rightarrow 
{\rm H}^0 \big (\omega_{\lfloor \frac{s}{2} \rfloor} \big) \otimes 
{\rm H}^0 \big (\omega_{\lceil \frac{s}{2} \rceil} \big),
$$
where $ \lfloor \frac{s}{2} \rfloor$ (resp.  $\lceil \frac{s}{2} \rceil$)
denotes the largest (resp. smallest) integer less (resp. larger) than 
$\frac{s}{2}$. This way we obtain a $G$-equivariant map
\begin{equation} 
\label{ten1}
 \bigotimes_{s=1}^{n-1} {\rm H}^0 \big (\omega_s \big)
\rightarrow \bigotimes_{s=1}^{n-1} 
\big({\rm H}^0 \big (\omega_{\lfloor \frac{s}{2} \rfloor} \big) \otimes 
{\rm H}^0 \big (\omega_{\lceil \frac{s}{2} \rceil} \big) \big) .
\end{equation}
Next we define  
\begin{equation} 
\label{ten2}
\bigotimes_{s=1}^{n-1} {\rm H}^0 \big (\omega_s \big)
\rightarrow \bigotimes_{s=1}^{n-1} 
\big({\rm H}^0 \big (\omega_{\lfloor \frac{s+n}{2} \rfloor} \big) \otimes 
{\rm H}^0 \big (\omega_{\lceil \frac{s+n}{2} \rceil} \big) \big) 
\end{equation}
as the tensorproduct of the maps ($1 \leq s <n$)
$$ \eta_{ \lfloor \frac{s+n}{2} \rfloor, \lceil \frac{s+n}{2} \rceil} :
{\rm H}^0 \big (\omega_s \big) \rightarrow 
{\rm H}^0 \big (\omega_{\lfloor \frac{s+n}{2} \rfloor} \big) \otimes 
{\rm H}^0 \big (\omega_{\lceil \frac{s+n}{2} \rceil} \big).
$$
We then notice
\begin{lem}
\label{steinb}
Any element in the image of (\ref{ten1}) or
(\ref{ten2}) will vanish with multiplicity at least
$\frac{n^2-2n}{4}$, when $n$ is even, and  
at least $\frac{(n-1)^2}{4}$, when n is odd,
along the diagonal in $G \times G$.
\end{lem}
\begin{proof}
For the map (\ref{ten1}) the statement 
follows from Lemma \ref{lemma1} and the formula
\begin{equation}
\label{formula}
\sum_{s=1}^{n-1} \lfloor \frac{s}{2} \rfloor = 
\begin{cases}
\frac{(n-1)^2}{4} & \text{ when $n$ odd,} \\
\frac{n^2-2n}{4} &  \text{ when $n$ even.} 
\end{cases}
\end{equation}
For (\ref{ten2}) we apply Lemma \ref{lemma2-1}
along with (\ref{formula}).
\end{proof}

\subsection{The construction of $\eta$}

\begin{lem}
There exists a nonzero $G$-equivariant map
\begin{equation}
\label{prod2}
{\rm St} \rightarrow \bigotimes_{s=1}^{n-1} 
{\rm H}^0 \big (\omega_s \big)^{\otimes (p-1)},
\end{equation}
which is uniquely determined up to nonzero constants.
\end{lem}
\begin{proof}
As ${\rm H}^0 \big (\omega_s \big)$ and  ${\rm H}^0 
\big (\omega_{n-s} \big)$, for $s=1, \dots,n-1$, are dual 
$G$-modules we start by observing that both sides of
(\ref{prod2}) are self-dual modules. Hence, maps of the 
form (\ref{prod2}) are in 1-1 correspondence with 
nonzero $G$-equivariant maps
\begin{equation}
\label{produ1} 
\bigotimes_{s=1}^{n-1} {\rm H}^0 \big (\omega_s \big)^{\otimes (p-1)}
\rightarrow {\rm St}.
\end{equation}
By Frobenius reciprocity the latter maps are constant 
multiples of the product map. This ends the proof. 
\end{proof}

We may now combine (\ref{prod2}) with the $(p-1)$-th 
tensorpower of (\ref{ten1}) and (\ref{ten2}). This leads to 
$G$-equivariant maps 
\begin{equation} 
\label{eta1}
\eta_1 : {\rm St} \rightarrow \bigotimes_{s=1}^{n-1} 
\big({\rm H}^0 \big (\omega_{\lfloor \frac{s}{2} \rfloor} \big) \otimes 
{\rm H}^0 \big (\omega_{\lceil \frac{s}{2} \rceil} \big) \big) ^{\otimes (p-1)},
\end{equation}
and 
\begin{equation}
\label{eta2}
 \eta_2 : {\rm St} \rightarrow \bigotimes_{s=1}^{n-1} 
\big({\rm H}^0 \big (\omega_{\lfloor \frac{s+n}{2} \rfloor} \big) \otimes 
{\rm H}^0 \big (\omega_{\lceil \frac{s+n}{2} \rceil} \big) \big) ^{\otimes (p-1)}.
\end{equation}
Next we take the tensorproduct of $\eta_1$ and $\eta_2$ 
with the $(p-1)$-th tensorpower of 
$$ \eta_{ \lfloor \frac{n}{2} \rfloor, \lceil \frac{n}{2} \rceil} :
k \rightarrow 
{\rm H}^0 \big (\omega_{\lfloor \frac{n}{2} \rfloor} \big) \otimes 
{\rm H}^0 \big (\omega_{\lceil \frac{n}{2} \rceil} \big),
$$
and arrive at a $G$-equivariant map
\begin{equation}
\label{eta0}
 {\rm St} \otimes {\rm St} 
\rightarrow 
\bigotimes_{s=1}^{2n-1} 
\big({\rm H}^0 \big (\omega_{\lfloor \frac{s}{2} \rfloor} \big) \otimes 
{\rm H}^0 \big (\omega_{\lceil \frac{s}{2} \rceil} \big) \big) ^{\otimes (p-1)}.
\end{equation}
To fix notation we assume that $\eta_1$ corresponds the left
factor of the tensorproduct $ {\rm St} \otimes {\rm St}$ in (\ref{eta0}).
Finally when applying the product map on each tensorfactor on the right 
hand side of (\ref{eta0}) we obtain  a map 
\begin{equation}
\label{id2}
{\rm St} \otimes {\rm St} \rightarrow
 {\rm H}^0(2(p-1) \rho) \otimes {\rm H}^0(2(p-1) \rho) .
\end{equation}
Now $\eta$ is defined from (\ref{id2}) by applying the identification 
(\ref{id1}).

\subsection{Maximal multiplicity}

By the next result and Lemma \ref{com} any element in the image of $\eta$ will be 
compatible with the diagonal.

\begin{lem}
 \label{mult}
Every element within the image of (\ref{id2}) will vanish with 
multiplicity at least $(p-1){\rm dim}(\nicefrac{G}{B})=(p-1)
{n \choose 2}$ 
along the diagonal in $\nicefrac{G}{B} \times \nicefrac{G}{B}$.
\end{lem}
\begin{proof}
Notice that by Lemma \ref{lemma1} every element within the image 
of  $\eta_{ \lfloor \frac{n}{2} \rfloor, \lceil \frac{n}{2} \rceil}$ will
vanish with multiplicity at least $ \lfloor \frac{n}{2} \rfloor$ along the 
diagonal. Thus by Lemma \ref{steinb} 
it suffices  to notice that
$$  \frac{n}{2} + 2 \frac{n^2-2n}{4} ={ n \choose 2},$$
and 
$$ 
\frac{n-1}{2} + 2 \frac{(n-1)^2}{4}= {n \choose 2}.$$
\end{proof}

As a consequence,  we may now consider $\eta$ as a map
 $$ \eta :{\rm St} \otimes {\rm St} \rightarrow {\rm End}_F \big 
(\nicefrac{G}{B} \times \nicefrac{G}{B} , {\rm diag}(\nicefrac{G}{B} ) \big)^{(1)}.$$

\subsection{A residual normal crossing}
 The only thing left is to prove that $\eta$ contains a Frobenius splitting
in its image. For this, we start by observing

\begin{lem}
\label{restrict}
Let $v_-$ denote a lowest weight vector in ${\rm St}$. Then 
the restriction of the function $\eta_1(v_-)$ to $\{ I_n \}  \times U^- $ 
coincides with  $\mathfrak f_{n-1}^{p-1}$ up a nonzero constant. Similarly, 
the restriction of $\eta_2(v_-)$ to $\{ I_n \}  \times U^- $ coincides
with $\mathfrak g_{n-1}^{p-1}$ up to a nonzero constant.
\end{lem}
\begin{proof}
Notice, first of all, that the map (\ref{prod2}) will map  
$v_-$ into the tensorproduct of the lowest weight spaces of 
the various ${\rm H}^0\big(
\omega_s\big)$. Thus by Lemma \ref{eval} and the
construction of $\eta_1$, the restricition of $\eta_1(v_-)$
to $\{ I_n \}  \times U^- $ will  coincide with the product
of the $(p-1)$-th powers of 
$$  v_{(\lfloor \frac{s}{2} \rfloor +1, \dots, s)},$$
for $s=1,2, \dots, n-1$. The latter product is by definition equal to
$\mathfrak f_{n-1}^{p-1} $.

To obtain the statement about $\eta_2(v_-)$ one first 
recalls that $\eta_{i,j}$, for $i+j>n$, was defined as 
the $\sigma$-twist of $\eta_{n-i,n-j}$. Hence, by Lemma 
\ref{twist}, the element $\eta_2(v_-)$ must equal the
composition $\eta_1(v_-) \circ \tau$ up to a nonzero
constant.  The statement 
thus follows from the first part of the proof above.
\end{proof}

\begin{prop}
\label{residual1}
Let $\mathfrak f$ denote the image of  $v_- \otimes v_-$ under 
(\ref{id2}). Then the restriction of $\mathfrak f$  to  
$\{ I_n \}  \times U^- $ coincides 
with the product
$$(\mathfrak f_n \mathfrak g_{n-1}) ^{p-1},$$
up to a nonzero constant.
\end{prop}
\begin{proof}
By Lemma \ref{restrict} we only have use Lemma \ref{eval} to observe 
that the restriction of the image of 
$$\eta_{ \lfloor \frac{n}{2} \rfloor, \lceil \frac{n}{2} \rceil}
: k  \rightarrow 
{\rm H}^0 \big (\omega_{\lfloor \frac{n}{2} \rfloor} \big) \otimes 
{\rm H}^0 \big (\omega_{\lceil \frac{n}{2} \rceil} \big),
$$
to $\{ I_n \}  \times U^- $ is generated by 
$  v_{(\lfloor \frac{n}{2} \rfloor +1, \dots, n)}.$
\end{proof}

\subsection{Frobenius splitting}
We may now prove 

\begin{thm}
\label {thm1}
The $G$-equivariant map $\eta$ defines a canonical Frobenius 
splitting of $\nicefrac{G}{B} \times \nicefrac{G}{B}$.
\end{thm}
\begin{proof}
Consider the commutative diagram
\begin{equation}
\label{dia1}
\xymatrix{
{\rm St} \otimes {\rm St} \ar[d]_{(\ref{id2})} \ar[rrd]^{\phi} & & \\
{\rm H}^0(2(p-1) \rho) \otimes {\rm H}^0(2(p-1) \rho)
\ar[d]_(.5){ {p}_{ 2(p-1)\rho} \otimes 1} \ar[rr] & & k_{2(1-p)\rho}\\
k_{2(1-p)\rho} \otimes {\rm H}^0(2(p-1) \rho) 
\ar[urr]_{1 \otimes { \rm ev}}
& & 
}
\end{equation}
induced by the projection $p_{ 2(p-1)\rho}$ onto the 
lowest weight space and the  $G$-equivariant evaluation map  
$$ {\rm ev} :  {\rm H}^0( 2(p-1)\rho) = 
{\rm End}_F\big(\nicefrac{G}{B}\big)^{(1)} \rightarrow
k.$$
We claim that $\phi(v_- \otimes v_-)$ is nonzero. 
To see this let $1 \otimes s$ denote the image of
$v_- \otimes v_-$ under the composed vertical map 
in (\ref{dia1}). We have to prove that ${\rm ev}(s)$
is nonzero. This can be checked locally on $U^-$ 
(considered as an open subset of $\nicefrac{G}{B}$)
where it follows  by Proposition \ref{residual1}, 
Proposition \ref{res4} and the
discussion in Section \ref{connection}.

By construction $\phi$
is the composition of the $G$-equivariant map
\begin{equation}
\label{phi1}
{\rm St } \otimes {\rm St} \xrightarrow{(\ref{id2})} 
 {\rm H}^0( 2(p-1)\rho) \otimes {\rm H}^0( 2(p-1)\rho)
\xrightarrow{1 \otimes {\rm ev} } {\rm H}^0( 2(p-1)\rho),
\end{equation}
with the projection map $p_{ 2(p-1)\rho}$ . Moreover, 
by Frobenius reciprocity the composed map (\ref{phi1}) is determined 
up to a  constant. In particular, if $\phi$ is nonzero
then (\ref{phi1}) must coincide with the multiplication 
map 
$$ {\rm H}^0( (p-1)\rho)  \otimes {\rm H}^0( (p-1)\rho) 
\rightarrow {\rm H}^0( 2(p-1)\rho),$$
up to a nonzero constant. As a consequence (\ref{phi1})
is surjective and thus it contains a Frobenius 
splitting of $\nicefrac{G}{B}$ in its image. Hence
(\ref{id2}), and thus also  $\eta$, must contain a Frobenius 
splitting in its image.
\end{proof}

\begin{remark}

The canonical Frobenius splitting $\eta$ is related 
to the concrete Frobenius splitting of $\nicefrac{G}{B} \times 
\nicefrac{G}{B}$ described in Section 4 of \cite{LT}.
More precisely, in [Sect.4, loc.cit] one defines for $g, h \in 
G$ a new matrix $M(g,h)$ of size $2n \times 2n$. 
Next one considers the principal minors $\delta_i(M(g,h))$ 
of size $i$ (from the lower left hand corner) which defines
sections of 
$$
{\rm H}^0 \big (\omega_{\lceil \frac{i}{2} \rceil} \big) \otimes 
{\rm H}^0 \big (\omega_{\lfloor \frac{i}{2} \rfloor} \big).
$$
The concrete Frobenius splitting of  $\nicefrac{G}{B} \times 
\nicefrac{G}{B}$ considered in [loc.cit] is then defined as 
the $(p-1)$-th power of the product 
$$ \prod_{i=1}^{2n-1} \delta_i(M(g,h)).$$
One may check that the elements $\delta_i(M(g,h))$ are 
connected to the above setup in the following way : 
when $1 \leq i \leq n$ the element  $\delta_i(M(g,h))$ generates the
image of the highest weight space under the map
$$\eta_{ \lceil \frac{i}{2} \rceil, \lfloor \frac{i}{2} \rfloor} :
 {\rm H}^0 \big (\omega_{i} \big) \rightarrow 
{\rm H}^0 \big (\omega_{\lceil \frac{i}{2} \rceil} \big) \otimes 
{\rm H}^0 \big (\omega_{\lfloor \frac{i}{2} \rfloor} \big).$$
When $n < i \leq 2n-1$  the element  $\delta_i(M(g,h))$ 
generates the image of the lowest weight space 
under the map
$$\eta_{ \lceil \frac{i}{2} \rceil, \lfloor \frac{i}{2} \rfloor} :
 {\rm H}^0 \big (\omega_{i-n} \big) \rightarrow 
{\rm H}^0 \big (\omega_{\lceil \frac{i}{2} \rceil} \big) \otimes 
{\rm H}^0 \big (\omega_{\lfloor \frac{i}{2} \rfloor} \big).$$
As a consequence, the concrete splitting in [loc.cit] coincide,
up to a flipping of the factors of  $\nicefrac{G}{B} \times 
\nicefrac{G}{B}$, with $\eta({v_+ \otimes v_-})$ for some
highest weight vector $v^+$ of ${\rm St}$. In particular,
 $\eta({v_+ \otimes v_-})$ is compatible with 
with subvarieties of the form $X \times X$ 
for $X$ a Kempf variety in $\nicefrac{G}{B}$.
\end{remark}

\section{The symplectic case}
  
In this section we will construct a canonical Frobenius 
splitting 
$$ \overline{\eta} :{\overline{\rm St}} \otimes
 {\overline{\rm St}}
\rightarrow {\rm End}_F\big(\nicefrac{\overline G}{\overline B}
\times \nicefrac{\overline G}{\overline B}, {\rm diag}(\nicefrac{\overline G}{\overline B})
 \big)^{(1)} .$$
of $\nicefrac{\overline G}{\overline B}
\times \nicefrac{\overline G}{\overline B}$ where any element in 
the image vanishes with multiplicity at least $(p-1) {\rm dim}(
\nicefrac{\overline G}{\overline B})$ along the diagonal.
We start by constructing a certain residual normal crossing on 
$\overline U^-$.

\subsection{The maps $\overline \eta_{i,j}$}
Notice  that the restrictions $\overline{\omega}_i$ and
$\overline{\omega}_{n-i}$ of the fundamental characters
$\omega_i$ and $\omega_{n-i}$
to $\overline{T}$  coincide. It follows that 
the $G$-equivariant maps $\eta_{i,j}$ in 
(\ref{embedding}) induces similar maps in the symplectic
case. E.g when $1 \leq i,j \leq 2m$ and $i+j \leq 2m$ we
obtain, by restriction, an $\overline G$-equivariant morphism
$$ \overline{\eta}_{i,j} : {\rm H}^0(\omega_{i+j}) 
\rightarrow {\rm H}^0 (\overline{\omega}_i ) 
\otimes {\rm H}^0 (\overline{\omega}_j ) .$$
Elements in the image of $\overline{\eta}_{i,j}$ 
may be considered  as functions on $\overline G \times \overline G$. 
Actually these functions are, by definition, restrictions 
of certain functions on $G \times G$ which were described 
in Section \ref{vanish}. In particular,  by applying Lemma 
\ref{lemma1} and Lemma \ref{lemma2-1}, we may obtain 
a lower bound on the vanishing 
multiplicity along the diagonal in $\overline G \times \overline G$; e.g. with the bound $i+j \leq 2m$
the multiplicity is at least ${\rm min}(i,j)$.

\subsection{Vanishing multiplicities}

As in the case of $\eta$ the map $\overline{\eta}$ 
will be constructed using certain tensorproducts 
of the $\overline{\eta}_{i,j}$. In this
case we consider the product of the maps
$$ \overline{\eta}_{{\lfloor \frac{s}{2} \rfloor} ,
{\lceil \frac{s}{2} \rceil}} : {\rm H}^0 \big(
\omega_s \big) \rightarrow 
{\rm H}^0 \big ( \overline{\omega}_{\lfloor \frac{s}{2} \rfloor} \big) \otimes 
{\rm H}^0 \big (\overline{\omega}_{\lceil \frac{s}{2} \rceil} \big) ,
$$
for $s=1, 2, \dots, 2m$, and 
$$  \overline{\eta}_{m ,0} : {\rm H}^0 \big(
\omega_m \big) \rightarrow 
{\rm H}^0 \big ( \overline{\omega}_{m} \big) \otimes k .$$
This leads to an $\overline G$-equivariant morphism (remember
$\omega_{2m}$ is the trivial character)
\begin{equation}
\label{map}
{\rm H}^0 \big(\omega_m \big) \otimes 
 \bigotimes_{s=1}^{2m-1} 
{\rm H}^0 \big(\omega_s \big)
\rightarrow 
\bigotimes_{s=1}^m {\rm H}^0 \big(
 \overline{\omega}_s \big)^{\otimes 2} \otimes 
\bigotimes_{s=1}^m {\rm H}^0 \big(
 \overline{\omega}_s \big)^{\otimes 2},
\end{equation}
The vanishing multiplicity along the diagonal in
$\overline G \times \overline G$ for elements in the image of (\ref{map}),
is then described by the following result.

\begin{lem}
\label{mult2}
Every element in the image of (\ref{map}) vanishes with 
multiplicity at least $m^2={\rm dim}(\nicefrac{\overline G}{\overline B})$
along the diagonal in $\overline G \times \overline G$. 
\end{lem}
\begin{proof}
Applying Lemma \ref{lemma1} it  suffices to observe the
formula 
$$ \sum_{s=1}^{2m}{  \lfloor \frac{s}{2} \rfloor}=m^2.$$
This ends the proof.
\end{proof}

\subsection{The construction of $\overline \eta$}

We now want to construct the canonical Frobenius splitting
$\overline \eta$. For this we first need to combine (\ref{map}) with the following 
lemma. 

\begin{lem} 
\label{stein2}
There exist nonzero $\overline G$-equivariant maps
\begin{equation}
\label{st1}
{\overline{\rm St}} \rightarrow 
\bigotimes_{s=1}^{m} {\rm H}^0 \big(\omega_s \big)^{\otimes (p-1)},
\end{equation}
and 
\begin{equation}
\label{st2}
{\overline{\rm St}} \rightarrow 
\bigotimes_{s=m}^{2m-1} {\rm H}^0 \big(\omega_s \big)^{\otimes (p-1)},
\end{equation}
both of them uniquely defined up to nonzero constants.
\end{lem}
\begin{proof}
By dualizing the picture and using the selfduality of 
${\overline{\rm St}}$ we, first of all,  have to show the existence 
and uniqueness, up to constants, of  a nonzero $\overline G$-equivariant map
\begin{equation}
\label{Map}
\bigotimes_{s=1}^{m} {\rm H}^0 \big(\omega_{n-s} \big)^{\otimes (p-1)}
\rightarrow 
{\overline{\rm St}} = {\rm Ind}_{\overline B}^{\overline G}\big(-
\sum_{s=1}^m  (p-1)\overline{\omega}_i \big).
\end{equation}
As $\overline{\o_i}= \overline{\o_{n-i}}$, for 
$i=1,2,\dots,n-1$, 
this follows from Frobenius reciprocity and the remarks
in Section \ref{restrict2}.
The second part of 
the statement follows in the same way.
\end{proof}

Lemma \ref{stein2} enables us to compose the $(p-1)$-th 
tensorpower of  (\ref{map})  with the two morphisms in 
Lemma \ref{stein2}. This way we obtain an $\overline G$-equivariant 
morphism
\begin{equation}
\label{map2}
{\overline{\rm St}} \otimes {\overline{\rm St}}
\rightarrow 
\bigotimes_{s=1}^m {\rm H}^0 \big(
 \overline{\omega}_s \big)^{\otimes 2(p-1)} \otimes 
\bigotimes_{s=1}^m {\rm H}^0 \big(
 \overline{\omega}_s \big)^{\otimes 2(p-1)}.
\end{equation}
Actually as the left hand side of (\ref{map}) contains
two copies of the module ${\rm H}^0 \big( 
\omega_m \big)$ there is some ambiguity about 
the definition of (\ref{map2}). There are several
natural ways to construct (\ref{map2}) and 
all of them work equally well in the following.  
However,  to be precise
we fix the setup in such a way that the map 
(\ref{st2}) is associated to the copies of   
${\rm H}^0 \big( \omega_m \big)$ coming
from the map $ \overline{\eta}_{m ,0}$.

Composing (\ref{map2}) with the product 
morphism on each tensorfactor we 
next obtain the $\overline{G}$-equivariant
map 
\begin{equation}
\label{idd2}
{\overline{\rm St}} \otimes {\overline{\rm St}}
\rightarrow  {\rm H}^0 \big( \sum_{s=1}^m
2(p-1) \overline{\omega}_s \big) \otimes 
 {\rm H}^0 \big( \sum_{s=1}^m
2(p-1) \overline{\omega}_s \big),
\end{equation}
which by the relation 
$$  {\rm H}^0 \big( \sum_{s=1}^m
2(p-1) \overline{\omega}_s \big) \simeq 
 {\rm End}_F\big(\nicefrac{\overline G}{\overline B}
 \big)^{(1)} ,$$
then defines 
the  $\overline G$-equivariant  map of primary interest 
$$ \overline{\eta} :{\overline{\rm St}} \otimes
 {\overline{\rm St}}
\rightarrow {\rm End}_F\big(\nicefrac{\overline G}{\overline B}
\times \nicefrac{\overline G}{\overline B}
 \big)^{(1)} .$$

\subsection{A residual normal crossing}

Observe that an element $g=(g_{i,j})$ in $\overline U^-$ 
is uniquely determined by the entries $g_{i,j}$ for 
$1 \leq j < i \leq 2m$ and $ i+j \leq n+1$. This follows 
directly from the relation (\ref{rel1})
satisfied by elements in ${\overline G}$. 
In particular, we may identify
$\overline U^-$ with affine $ m^2$-space 
${\mathbb A}^{m^2}$ through the map
\begin{equation}
\label{ide1} 
\overline  U^- \rightarrow \mathbb A^{m^2},
\end{equation}
$$ g \mapsto (g_{i,j})_{\stackrel{1 \leq j < i \leq 2m}{i+j \leq n+1}},$$
where we, for convenience, have indexed the 
coordinates in $\mathbb A^{m^2}$ by the 
set of pairs $(i,j)$ satisfying  
$1 \leq j < i \leq 2m$ and $ i+j \leq n+1$.
The coordinate ring of ${\mathbb A}^{m^2}$ 
is then identified with $k[z_{i,j}]$ accordingly. 
Using the isomorphism  (\ref{ide1}) we also 
consider the coordinate ring  $\overline U^-$
as  $k[z_{i,j}]$.

Recall that we in Section \ref{residual} introduced
a collection $\mathfrak f_r$, $1 \leq r \leq n$, of
polynomial functions on $U^{-}$. In the following
$\overline{\mathfrak f}_r$, $1 \leq r \leq n$, will
denote their restriction to 
 $\overline{U}^-$.

\begin{prop}
\label{ressymp}
The polynomial 
$\overline{\mathfrak f}_r \in k[z_{i,j}]$  
is a residual normal crossing relative
to the set of variables 
$$ \{ z_{i,j} :  1 \leq j < i \leq n, ~ i+j \leq r+1 \},$$
which are ordered using the order : $ z_{i,j} \leq
z_{i',j'}$ if either $i+j < i'+j'$ or  $i+j = i' + j'$
and $i \leq i'$.   
\end{prop}
\begin{proof}
Recall the notation $k[Z_{i,j}]$, $1 \leq j < i \leq n$, 
for the coordinate ring of $U^-$ in Section \ref{residual}.
On the level of coordinate rings the inclusion $\overline U^- 
\subset U^-$ then takes the form
\begin{equation}
\label{restr1}
k[Z_{i,j}] \rightarrow k[z_{i,j}],
\end{equation}
where
$$ Z_{i,j} \mapsto
\begin{cases}
z_{i,j} & \text{if $ i+j  \leq n+1$,} \\
f_{i,j} & \text{else,} \\
\end{cases}
$$
for certain elements $f_{i,j}$ in $k[z_{i,j}]$. Now
 we may apply Lemma \ref{res1} and 
Lemma \ref{subst} to end the proof.
\end{proof}

\subsection{Frobenius splitting}

Let $\overline{v}_-$ denote a lowest weight vector of 
 $\overline{\rm St}$. Before proving the main result 
we note  the following result 

\begin{lem}
\label{residual3}
Let $\overline{\mathfrak f}$ denote the image of
$\overline v_- \otimes \overline v_-$ under the 
map (\ref{idd2}). Then the restriction of 
$\overline{\mathfrak f}$ to $\{ I_n \} \times
\overline U^-$ coincide with $(\overline{\mathfrak
f}_n)^{p-1}$.
\end{lem}
\begin{proof}
Notice, first of all, that the maps (\ref{st1})  and (\ref{st2}) 
will map $\overline v_-$ into the tensorproduct of the lowest 
weight spaces of  the various ${\rm H}^0\big(
 \omega_s\big)$. Thus by Lemma \ref{eval} and the
construction of (\ref{idd2}), the restriction of 
$\overline{\mathfrak f}$ to $\{ I_n \} \times
\overline U^-$ will  coincide with the product
of the $(p-1)$-th powers of 
$$  \overline v_{(\lfloor \frac{s}{2} \rfloor +1, \dots, s)},$$
for $s=1,2, \dots, 2m$. Here  $\overline 
v_{(\lfloor \frac{s}{2} \rfloor +1, \dots, s)}$ denotes the
restriction of $ v_{(\lfloor \frac{s}{2} \rfloor +1, \dots, s)}$
to $\overline U^-$.
The latter product is by definition equal to
$(\overline{\mathfrak f}_{n})^{p-1} $.

\end{proof}

\begin{thm}
\label{thm2}
The map $\overline{\eta}$ defines a canonical Frobenius splitting
of $\nicefrac{\overline G}{\overline B}
\times \nicefrac{\overline G}{\overline B}$. Moreover, 
every element in the image of $\overline{\eta}$ vanishes
with multiplicty $(p-1) {\rm dim} (\nicefrac{\overline G}{\overline B})$
along the diagonal. In particular, we may consider $\overline{\eta}$ 
as a map
\begin{equation}
\label{etasymp}
\overline{\eta} :{\overline{\rm St}} \otimes
 {\overline{\rm St}}
\rightarrow {\rm End}_F\big(
\nicefrac{\overline G}{\overline B}
\times \nicefrac{\overline G}{\overline B}
,
{\rm diag}(\nicefrac{\overline G}{\overline B}) \big) ^{(1)},
\end{equation}
where any Frobenius splitting in the image vanishes
with maximal multiplicity along the diagonal.
\end{thm}
\begin{proof}
By Lemma \ref{mult2} we only have to prove that $\overline{\eta}$
contains a Frobenius splitting in its image. For this we proceed as in the proof
of Theorem \ref{thm1}. Start by fixing a lowest weight vector $\overline{v}_-$  
in $\overline{\rm St}$ and consider the diagram ($\mu$ denoting $\sum_{s=1}^m2 (p-1) 
\overline{\omega}_s$)
\begin{equation}
\label{dia2}
\xymatrix{
{\overline{\rm St}} \otimes \overline{{\rm St}} \ar[d]_{(\ref{idd2})} \ar[rrd]^{\overline \phi} & & \\
{\rm H}^0(\mu) \otimes {\rm H}^0(\mu)
\ar[d]_(.5){ {p}_{ \mu} \otimes 1} \ar[rr] & & k_{-\mu}\\
k_{-\mu} \otimes {\rm H}^0(\mu) 
\ar[urr]_{1 \otimes { \overline {\rm ev}}}
& & 
}
\end{equation}
induced by the projection $p_{\mu}$ onto the 
lowest weight space and the  $\overline G$-equivariant 
evaluation map  
$$ {\overline{\rm ev}} :  {\rm H}^0( \mu) = 
{\rm End}_F\big(\nicefrac{\overline G}{\overline B}\big)^{(1)} \rightarrow
k.$$
As in the proof of Theorem \ref{thm1} it suffices to 
prove that $\overline \phi(\overline v_- \otimes \overline v_-)$ 
is nonzero. To prove this let $1 \otimes \overline s$ denote the 
image of $\overline v_- \otimes \overline v_-$ under the 
composed vertical map in (\ref{dia2}). We have to prove 
that ${\overline{\rm ev}}(\overline s)$
is nonzero. This can be checked locally on $\overline U^-$ 
(considered as an open subset of $\nicefrac{\overline G}{
\overline B}$)
where it follows  by Lemma \ref{residual3}, 
Proposition \ref{ressymp} and the
discussion in Section \ref{connection}.

\end{proof}

\bibliographystyle{amsalpha}

\end{document}